\documentclass{amsart}
\usepackage{enumerate}
\usepackage{calc}
\usepackage[dvips]{graphicx}
\usepackage{color}
\usepackage{pstricks}
\usepackage{epigraph}
\usepackage{amsmath, amssymb}
\usepackage{pst-node}
\newtheorem{theo}{Theorem}[section]

\newtheorem{cor}[theo]{Corollary}

\newtheorem{lem}[theo]{Lemma}

\newtheorem{prop}[theo]{Proposition}

\newtheorem{defn}[theo]{Definition}

\theoremstyle{definition}
\newtheorem{rem}[theo]{Remark}

\def\a{{\a}}
\def\a{{\mathfrak a}}

\def\C{\mathbb C}

\def\E{\mathbb E}

\def\N{\mathbb N}
\def\P{\mathbb{P}}

\def\Q{\mathbb Q}
\def\R{\mathbb R}
\def\rad{\operatorname{rad}}

\def\Z{\mathbb{Z}}

\DeclareMathOperator{\Ad}{Ad} 
\DeclareMathOperator{\Span}{span}
\DeclareMathOperator{\Map}{Map}

\DeclareMathOperator{\ch}{ch}
\begin{document} 

\title[Kirillov-Frenkel character formula for loop groups]{Kirillov-Frenkel character formula for  Loop groups, radial part and Brownian sheet}

\author{Manon Defosseux}
\address{Laboratoire de Math\'ematiques Appliqu\'ees \`a Paris 5, Universit\'e Paris 5, 45 rue des  Saints P\`eres, 75270 Paris Cedex 06.}
\email{manon.defosseux@parisdescartes.fr}
%    General info
%\date{a}
\maketitle
  \epigraph{Space is a swarming in the eyes, and Time a singing in the ears.}{Vladimir Nabokov, \textit{Ada or Ardor: A Family Chronicle}}   
This paper is dedicated to Professor Philippe Bougerol to thank him for many  discussions which have been a great source of joy and inspiration to me, and a valuable aid especially in the work presented here.

\begin{abstract}  We consider the coadjoint action of a Loop group of a compact group on the dual of the corresponding centrally extended Loop algebra and prove that a Brownian motion in a Cartan subalgebra conditioned to remain in an affine Weyl chamber - which can be seen as a space time conditioned Brownian motion - is distributed as the radial part process of a Brownian sheet on the underlying Lie algebra.

\end{abstract}
\section{Introduction} 
It is a famous result that a real Brownian motion conditioned in Doob's sense  to remain positive, is distributed as a Bessel $3$ process, i.e. as the radial process of a $3$-dimensional Brownian motion. More generally, if one considers the adjoint action of a compact Lie group on its Lie algebra, the radial part process of a Brownian motion in the Lie algebra is distributed as  a Brownian motion in a Cartan subalgebra  conditioned  to remain in a Weyl chamber. 
In this paper, we consider the coadjoint action of a Loop group of a compact group on the dual of the corresponding centrally extended Loop algebra and prove that a Brownian motion in a Cartan subalgebra conditioned to remain in an affine Weyl chamber - which can be seen as a space time conditioned Brownian motion - is distributed as the radial part process of a Brownian sheet on the underlying Lie algebra.

Let us be more precise. Let $K$ be a connected compact Lie group, $\mathfrak k$ its Lie algebra, and   $\mathfrak t$ a Cartan subalgebra. One considers  a   Weyl chamber in $\mathfrak t$. Then the orbits of $\mathfrak k$ under the adjoint action of $K$   are parametrized by the Weyl  chamber. Actually for any $x\in \mathfrak k$, it exists a unique vector in the Weyl chamber which is in the same orbit as $x$. This vector is called the radial part of $x$. The Lie algebra $\mathfrak k$ is equipped with an $\Ad(K)$-invariant scalar product, whose restriction to the Cartan subalgebra is invariant for the action of the Weyl group. If we consider the radial part process of  a standard Brownian motion  on $\mathfrak k$, it is a classical fact that this process is distributed as its projection on the Cartan subalgebra $\mathfrak t$ - which is a Brownian motion on $\mathfrak t$ - conditioned in Doob's sense to remain forever in the Weyl chamber. The Kirillov's character formula is at the heart of the connection between the two processes. 

An affine Lie algebra  can be realized as a central extension of a loop algebra $L\mathfrak k$.  Considering the coadjoint action of the  loop group $LK$ on the dual of the centrally extended loop algebra,  one defines the radial part of an element of this dual algebra (see Pressley and Segal \cite{Segal}). Frenkel established in \cite{Frenkel}  a Kirillov character type formula in the framework of  affine Lie algebras considering a Gaussian measure on the dual of $L\mathfrak k$ - basically a Brownian motion on $\mathfrak k$. In his approach, the conditional law of a Brownian motion on $\mathfrak k$ given its radial part provides a natural measure on the corresponding orbit under the action of $LK$.  It highly suggested that one could construct a process on a loop algebra, whose the radial part process would be distributed as a projection on a Cartan subalgebra,  conditioned to remain in an affine Weyl chamber. 

The affine Weyl chamber is a fundamental domain for the action of the Weyl group on the Tits cone. The first difficulty is that there is no Euclidean structure on the Cartan subalgebra of an affine Lie algebra, which would be invariant for the action  the Weyl group. So there is no natural Brownian motion to consider on it.  The Kirillov's orbit  method, on which we based our intuition, suggests a connection between coadjoint orbits for the action of a Loop group on the dual of an affine Lie algebra and irreducible representations of the affine Lie algebra.  Tensor product of irreducible representations of an affine Lie algebra makes appear a drift in the direction of the fundamental weight $\Lambda_0$. The idea is to consider a process with such a drift in the direction of  $\Lambda_0$ - which can be seen as a time component - living in an affine Lie algebra. In this paper, we construct such a process, considering a Brownian sheet on the Lie algebra  $\mathfrak k$ and prove that the corresponding radial  part process   is distributed as a projection on a Cartan subalgebra, conditioned to remain forever in an affine Weyl chamber. 

The paper is organized as follows.   In section \ref{Action orbits}  we describe the orbits for the coadjoint action of a Loop group of  a compact group on the dual of the associated centrally extended loop algebra. In particular we define a notion of radial part for this action, which is  suitable for our context. In section \ref{affine}   we briefly recall the necessary background on representation theory of affine Lie algebras.   Sections \ref{cond-law-end-point} and \ref{Kirillov-Frenkel} are basically a reformulation of the main results of \cite{Frenkel}. In section \ref{cond-law-end-point} we compute the  conditional law of a Brownian motion indexed by $[0,1]$  given the end point of its stochastic exponentiel. In section \ref{Kirillov-Frenkel} we recall how Frenkel proves that this conditional law leads to a Kirillov character formula for affine Lie algebras.    In section \ref{cond-affine} we introduce a Brownian motion on a Cartan subalgebra of an affine Lie algebra conditioned - in Doob's sense - to remain forever in an affine Weyl chamber.  We prove in section \ref{cond-rad} that this  conditioned Doob process has the same law as the radial part process of a Brownian sheet on $\mathfrak k$.

%\textit{Acknowlegments:} The author would like to thank Philippe Bougerol for having made her benefit  with constance %and kindness  from its deep point of view about the subject of the paper. It has been of particular importance for this paper.  

\section{Action of Loop group and its orbits}\label{Action orbits}
 \paragraph{\bf Loop group and its action} The following presentation is largely inspired by the one given in \cite{Segal}. Let $K$ be a connected, simply connected, compact Lie group and $\mathfrak{k}$ its Lie algebra. By compactness, without loss of generality, we suppose that $K$ is a matrix Lie group. The adjoint action of $K$ on itself, which is denoted by $\mbox{Ad}$,  is defined by $\mbox{Ad}(k)(u)=kuk^*$, $k,u\in K$. The induced adjoint action of  $K$ on its Lie algebra $\mathfrak{k}$ is still denoted by $\mbox{Ad}$ and is defined by $\mbox{Ad}(k)(x)=kxk^*$, $k\in K$, $x\in \mathfrak{k}$. The   Lie bracket on $\mathfrak{k}$ is denoted by $[.,.]_\mathfrak{k}$. We equipp $\mathfrak k$ with an $\mbox{Ad}(K)-$invariant inner product $(.,.)$, i.e.   the negative of the Killing form.  We denote by $e$ the identity of $K$. We consider the group $\Map([0,1],K)$ of  Borel measurable  maps from $[0,1]$ to $K$, the group law being pointwise composition, and    
the subgroup $LK$  of smooth loops from $[0,1]$ to $K$,  
$$LK=\{ f:[0,1]\to K,\,  f\textrm{ is smooth } ,\, f(0)=f(1)\}.$$
The Lie algebra $L\mathfrak{k}$ over $\R$ of $LK$ is the loop algebra  of smooth loops  from to $[0,1]$ to $\mathfrak k$. We define a centrally extended Lie algebra $L\mathfrak k\oplus \mathbb Rc$ equipped with a Lie bracket given by 
$$[\xi+\lambda c,\eta+\mu c]=[\xi,\eta]_\mathfrak{k}+\omega(\xi,\eta)c,$$ 
for $\xi,\eta\in L\mathfrak k$, $\lambda,\mu\in \R$, where $\omega$ is the cocycle defined by $\omega(\xi,\eta)=\int_0^{1}(\xi'(t),\eta(t))\, dt,$ and $[\xi,\eta]_\mathfrak{k}$ is defined pointwise. A Cartan subalgebra of the extended Loop algebra is $\mathfrak t\oplus \R c$, where $\mathfrak t$ is identified with the set of $\mathfrak t$-valued constant loops.  The Lie bracket   actually defines a Lie algebra action of $L\mathfrak k$ on $L\mathfrak k \oplus \R c$ given by 
$$\xi.(\eta+\mu c)=[\xi,\eta]+\omega(\xi,\eta)c,$$
for any $\xi\in L\mathfrak k,(\eta,\mu)\in L\mathfrak k\times \R$. This action comes from the adjoint action of $LK$ on  $L\mathfrak k \oplus \R c$ defined by 
\begin{align}\label{adjoint}
 \gamma.(\eta+\lambda c)=\mbox{Ad}(\gamma)(\eta)+\big(\lambda +\int_0^{1}(\gamma_s^{-1}\gamma_s',\eta_s)\, ds\big)c,
\end{align}
for any $\gamma\in L K,(\eta,\lambda)\in L\mathfrak k\times \R$.  
The corresponding coadjoint action of $LK$ on   the algebraic dual $(L\mathfrak{k}\oplus \R c)^*=(L\mathfrak{k})^*\oplus \R\Lambda_0$, where $\Lambda_0(L\mathfrak{k})=0$ and $\Lambda_0(c)=1$, is defined by
\begin{align}\label{coadjoint} \gamma.(\phi+\lambda \Lambda_0) =[Ad^*(\gamma)\phi-\lambda\int_0^{1}(\gamma_s'\gamma_s^{-1},.)\, ds]+\lambda\Lambda_0,
\end{align}
for any $\phi\in (L\mathfrak k)^*, \lambda\in \R$,
where $\int_0^{1}(\gamma_s'\gamma_s^{-1},.)\, ds$ stands for the linear form defined by $$x\in\mathfrak k\mapsto \int_0^{1}(\gamma_s'\gamma_s^{-1},x_s)\, ds.$$
Notice that the coadjoint action of the loop group doesn't affect the level, i.e. the coordinate in $\Lambda_0$.
 Let us equipp $L\mathfrak k$ with the   $L_2$-norm, and consider its completion $L_2([0,1],\mathfrak k)$ with respect with the $L_2$-norm. We define  
\begin{align*} 
LH^1([0,1],K)&=\{\gamma:[0,1]\to K, \gamma\textrm{ is absolutely continuous, }\, 
 \gamma^{-1}\gamma' \in L_2([0,1],\mathfrak k),\, \gamma(0)=\gamma(1) \}\\
H^1([0,1],\mathfrak k)&=\{f:([0,1]\to\mathfrak k, f\textrm{ is absolutely continuous}, f'\in L_2([0,1],\mathfrak k) ,\,  f(0)=0\},\\
H^1([0,1],K)&=\{h:[0,1]\to K, h\textrm{ is absolutely continuous, }\, h^{-1}h' \in L_2([0,1],\mathfrak k),\, h(0)=e \}.
\end{align*}
If we denote by $(L_2([0,1],\mathfrak k))'$  the topological dual of $L_2([0,1],\mathfrak k)$ then (\ref{coadjoint}) defines  an action of $LH^1([0,1],K)$  on  $(L_2([0,1],\mathfrak k))'\oplus \R\Lambda_0$, given by 
$$\gamma.(\phi_x+\lambda\Lambda_0) =\int_0^{1}(\gamma_sx'_s\gamma_s^{-1},.)\, ds-\lambda\int_0^{1}(\gamma_s' \gamma_s^{-1},.)\, ds+\lambda\Lambda_0,$$
for any $\gamma\in LH^1([0,1],K)$, where  $\phi_x$ is a linear form  defined on $L_2([0,1],\mathfrak k)$  by $$\phi_x(y)=\int_0^{1}(y_t,x'_t)\, dt,$$ for any $y\in L_2([0,1],\mathfrak k)$, and  $x\in H^1([0,1],\mathfrak k)$. This    action    gives rise to an action  $.$ of $LH^1([0,1],K)$ on $H^1([0,1],\mathfrak{k})\oplus \R\Lambda_0$ defined by 
\begin{align} \label{correspondence}
\gamma. (x+\lambda\Lambda_0)=\int_0^.(\gamma_s x'_s \gamma^{-1}_s\, - \lambda\gamma'_s\gamma^{-1}_s)\, ds+\lambda\Lambda_0,
\end{align}
for any  $\gamma\in LH^1([0,1],K)$, $x\in H^1([0,1],\mathfrak{k})$ and $\lambda\in\R$, which satisfies $$\phi_{(\gamma.(x+\lambda\Lambda_0)-\lambda\Lambda_0)}+ \lambda\Lambda_0=\gamma.(\phi_{x}+ \lambda\Lambda_0).$$
There is another way to make  this action appear naturally.  
One defines a group action of the loop group $LH^1([0,1],K)$ on $H^1([0,1],K)$ by letting for any $\gamma\in LH^1([0,1],K)$,
$$ (\gamma.h)_t=\gamma_0h_t\gamma_t^{-1},$$
for $h\in H^1([0,1],K)$ and $t\in [0,1]$.
The set $H^1([0,1],K)$ is in one to one correspondence with the set $H^1([0,1],\mathfrak k)$.
Actually for $\lambda\in \R^*$, this is a classical result of differential geometry that for any path $x\in H^1([0,1],\mathfrak{k})$ it exists a unique $X\in H^1([0,1],K)$ such that $\lambda\, dX=X\, dx$.  For any $x\in H^1([0,1],\mathfrak{k}),$ and $\lambda\in \R^*$, one defines $\epsilon(x+\lambda\Lambda_0)$ as the unique map in $H^1([0,1],K)$ which satisfies this differential equation. As 
$$\lambda\, d(\gamma_0\epsilon(x+\lambda\Lambda_0)\gamma^{-1})= \gamma_0\epsilon(x+\lambda\Lambda_0)\gamma^{-1}(\gamma dx\gamma^{-1}-\lambda d\gamma\gamma^{-1}),$$ the  action of $LH^1([0,1],K)$ on $H^1([0,1],K)$, and  the action of $LH^1([0,1],K)$ on $H^1([0,1],\mathfrak{k})\oplus \R\Lambda_0$, satisfy $$\epsilon(\gamma.(x+\lambda\Lambda_0))=\gamma.\epsilon(x+\lambda\Lambda_0)$$ when $\lambda\ne 0$.

\paragraph{\bf Roots and weights} 
 We choose a maximal torus $T$ of $K$ and denote  by $\mathfrak{t}$ its Lie algebra. 
 We denote by $\mathfrak{g}$ the complexification  of $\mathfrak{k}$, i.e. $\mathfrak g=\mathfrak k\oplus i\mathfrak k$.
 We consider  the set of real roots $$R=\{\alpha\in \mathfrak{t}^*: \exists X \in \mathfrak{g}\setminus \{ 0 \},\, \forall H\in \mathfrak{t} ,\, [H,X]=2i\pi\alpha (H)X\}.$$ We choose  a set $\Sigma$ of simple roots of $R$ and denote by $R_+$ the set of positive roots.  The half sum of positive roots is denoted by $\rho$. Letting for $\alpha\in R$, $$\mathfrak{g}_\alpha=\{X\in \mathfrak{g}: \, \forall H\in \mathfrak{t},\, [H,X]=2i\pi\alpha(H)X\},$$ the coroot $\alpha^\vee$  of $\alpha\in \Sigma$ is defined to be the only vector of $\mathfrak{t}$ in $[\mathfrak{g}_\alpha,\mathfrak{g}_{-\alpha}]$ such that $\alpha(\alpha^\vee)=2$. The dual Coxeter number denoted by  $h^\vee$ is equal to $1+\rho(\theta^\vee)$, where $\theta^\vee$ is the highest coroot.   We denote respectively by $Q=\sum_i\Z\alpha_i$ and $Q^\vee=\sum_i\Z\alpha_i^\vee$ the root and the coroot lattice.      The  weight lattice $ \{\lambda\in \mathfrak{t}^*: \lambda(\alpha^\vee)\in \mathbb{Z}, \, \forall \alpha\in \Sigma\} $ is denoted by $P$ and the set $\{\lambda\in \mathfrak{t}^*: \lambda(\alpha^\vee)\in \mathbb{N}, \, \forall\alpha\in \Sigma\} $ of dominant weights is denoted by $P_+$.
\newline 

\paragraph{\bf Orbits and radial part of a   path from $H^1([0,1],\mathfrak k)$.}  
\begin{prop} Let $x,y\in H^1([0,1],\mathfrak{k})$ and $\lambda\in \R_+^*$. 
\begin{enumerate}
\item  For any $ \gamma\in LH^1([0,1],K)$, one has  $\gamma.(x+\lambda\Lambda_0)=(y+\lambda\Lambda_0)$ if and only if  $\gamma.(\phi_{x}+\lambda\Lambda_0)=\phi_{y}+\lambda\Lambda_0.$
\item It exists
$ \gamma\in LH^1([0,1],K) ,$ such that $\gamma.(x+\lambda\Lambda_0)=y$ if and only if  it exists $u\in K$ such that $\Ad(u)\epsilon(x+\lambda\Lambda_0)_{1}=\epsilon(y+\lambda\Lambda_0)_{1}.$
\end{enumerate}
\end{prop}
\begin{proof} The first point comes from identity (\ref{correspondence}). For the second we write  that   if $\Ad(u)\epsilon(x+\lambda\Lambda_0)_{1}=\epsilon(y+\lambda\Lambda_0)_{1}$, and   $\gamma=\epsilon(y+\lambda\Lambda_0)u\epsilon(x+\lambda\Lambda_0)^{-1},$ then $\gamma\in LH^1([0,1],K)$ and $\gamma.(x+\lambda\Lambda_0)= y+\lambda\Lambda_0$.
\end{proof}  
For $\alpha\in \Sigma$ the fundamental reflection $s_{ \alpha^\vee}$ is defined on $\mathfrak{t}$ by $$s_{\alpha^\vee}(x)=x-\alpha(x)\alpha^\vee,\quad  \textrm{ for } x \in \mathfrak{t}.$$ We consider the extended affine Weyl group  generated by the reflections $s_{\alpha^\vee}$ and the translations by $\alpha^\vee$, $x\in\mathfrak t\mapsto x+ \alpha^\vee,$ for $\alpha\in \Sigma$.  The fundamental domain   for its action  on $\mathfrak{t}$ is  
\begin{align*}
A=\{x\in \mathfrak{t} : \forall \alpha\in R_+, \, \, 0\le \alpha(x)\le  1\}
\end{align*}
(see for instance section 4.8 of \cite{Humphreys}).
For $x\in K$, one defines $\mathcal O_x$ as the adjoint orbit through $x$, i.e.
$$\mathcal O_x=\{y\in K: \exists u\in K, \, y=uxu^*\}.$$
We consider the exponential map $\exp:\mathfrak{k}\to K$.   As   $K$ is simply connected,  the set of conjugacy classes $K/\Ad(K)$ is in one-to-one correspondence with the fundamental domain $A$, i.e. for all $u\in K$,  it exists a unique  $x\in A$ such that  $u\in \mathcal O_{\exp(x)}$
 (see \cite{Brocker} for instance). Thus given $\lambda\in \R^*_+$,  every paths in $H^1([0,1],\mathfrak k)$ is conjugated to a  straight path $$s\in[0,1]\mapsto sy,$$ for some  $y\in A$, and one can give the following definition for the radial part of $x\in H^1([0,1],\mathfrak k)$, given a positive level $\lambda\in \R_+^*$.
\begin{defn} For $(x,\lambda)\in H^1([0,1],\mathfrak k)\times \R_+^*$, one defines the radial part of $x+\lambda\Lambda_0$ as the unique element $r$ in $A$ such that $\epsilon(x+\lambda\Lambda_0)_{1}\in \mathcal{O}_{\exp( r)}$. It is denoted by $\rad(x+\lambda\Lambda_0)$.
\end{defn} 
\paragraph{\bf Radial part of a continuous semi-martingale}  The aim of this part is to define the radial part of a $\mathfrak k$-valued Brownian path. Such a path is not in $H^1([0,1],\mathfrak k)$ but one can define a stochastic exponential of a Brownian path, which allows us to define its radial part, by analogy with what we've done above.   Let $(\Omega,\mathcal F,(\mathcal F_s)_{s\in[0,1]},\P)$ be a  filtered probability space.   The following results have been proved for instance in   \cite{karandikar} or \cite{lepingle}. If $(x_s)_{s\in [0,1]}$ is an $\mathfrak k$-valued continuous semi-martingale and $\lambda\in\R_+^*$, then the stochastic differential equation
\begin{align}\label{eds}
\lambda \, dX=X\circ dx,
\end{align} 
where $\circ$ stands for the Stratonovitch integral, has a unique solution starting from $e$. Such a solution is a $K$-valued process, that we still denote by $(\epsilon(x+\lambda\Lambda_0)_s)_{s\in [0,1]}$. This is  the Stratonovitch  stochastic exponential of $\frac{x}{\lambda}$.   Note that given $\lambda\in \R_+^*$, $x$ can be recovered from $\epsilon(x+\lambda\Lambda_0)$ as $dx=\lambda\, \epsilon(x+\lambda\Lambda_0)^{-1}\circ d\epsilon(x+\lambda\Lambda_0)$.  
 \begin{defn}  
For $\lambda\in \R_+^*$, and $x=(x_s)_{s\in[0,1]}$ a continuous $\mathfrak k$-valued semi-martingale, one defines the radial part of $(x+\lambda\Lambda_0)$  as the unique element $r$  in $A$ such that $\epsilon(x+\lambda\Lambda_0)_{1}\in \mathcal{O}_{\exp( r)}$. It is denoted by $\rad(x+\lambda\Lambda_0)$.
\end{defn}

 \section{The Conditional law of a brownian motion in $\mathfrak k$ given the end-point of its stochastic exponential}\label{cond-law-end-point}
We denote by $\Lambda$ the kernel of the restriction $\exp_{\vert \mathfrak{t}}$ and by $\Lambda^*$ the set of integral weights $ \{\lambda\in \mathfrak{t}^*: \lambda(\Lambda)\in  \mathbb{Z}\} $, which is included in $P$ since $\alpha^\vee\in \Lambda$ (see \cite{Brocker} for instance).  Thus, we define an application  $\vartheta_\lambda$  on $T$, when $\lambda\in \Lambda^*$, by letting 
 $\vartheta_\lambda(\exp(x))=e^{2i\pi\lambda(x)}$, for $x\in \mathfrak{t}$. The irreducible representations of $K$ are parametrized by the set $\Lambda^*_+=\Lambda^*\cap \mathcal{C}^\vee$, where $\mathcal C^\vee=\{\lambda\in \mathfrak t^*: \lambda(\alpha^\vee)\ge 0, \alpha\in \Sigma\}$.  Here, $K$ is supposed to be simply connected, so that $P=\Lambda^*$ and $P_+=\Lambda^*_+$. We denote by   $\ch_\lambda$ the character  of  the irreducible  representation with highest weight $\lambda\in  P_+$.  We recall that $\mathfrak t$ is equipped with a $W$-invariant scalar product $(.,.)$. We identify and $\mathfrak t$ and $\mathfrak t^*$ via   $(.,.)$ and  still denote by $(.,.)$ the   scalar product on $\mathfrak t^*.$  
We write sometimes $e^x$ instead of $\exp(x)$, for $x\in \mathfrak k$.
 
\paragraph{\bf Brownian motion on $K$}Let $(\Omega,(\mathcal F_s)_{s\in[ 0,1]},\mathbb P)$ be a probability space, where $(\mathcal F_s)_{s\in[ 0,1]}$ is the natural filtration of a  $\mathfrak{k}$-valued  standard continuous Brownian motion $(x_s^\sigma)_{s\in[0,1]}$  defined on $\Omega$,  with variance $\sigma>0$.  In this section we consider the stochastic exponential $\epsilon(x^\sigma+\lambda\Lambda_0)$ only for $\lambda=1$. As for any $\lambda\in \R_+^*$, $\epsilon(x^\sigma+\lambda\Lambda_0)=\epsilon(x^{\frac{\sigma}{\lambda^2}}+\Lambda_0)$, there is no loss of generality. In the sequel we let $\epsilon(x^\sigma)=\epsilon(x^\sigma+\Lambda_0)$. The stochastic exponential  $(\epsilon(x^\sigma)_s)_{s\in [0,1]}$ is a left Levy process on $K$ starting from $e$, with transition probability $(p_s^\sigma)_{s\in [0,1]}$ with respect to the Haar measure on $K$, defined on $K$ by
\begin{align*}
p^\sigma_s(x,y)=p_s^\sigma(e,x^{-1}y)=\sum_{\lambda\in P_{+}}\ch_\lambda(e)\ch_\lambda(x^{-1}y)e^{-\frac{s\sigma (2\pi )^2}{2} (\vert\vert\lambda+\rho\vert\vert^2 -\vert\vert \rho\vert\vert^2)},
\end{align*}
$s\in[0,1]$, $x,y\in K$.  In the sequel we write $p_s^\sigma(x)$ instead of $p_s^\sigma(e,x)$. This process is a Brownian motion on $K$. The following proposition states a Girsanov formula for a Brownian motion on a compact Lie group. It is proved  for instance in \cite{Driver}, \cite{Gordina} or \cite{karandikar}.  For a $\mathfrak k$-valued $L_2$ function $y$, and a $\mathfrak k$-valued continuous semi-martingale $(x_s)_{s\in [0,1]}$, $\int_0^t(y_s,dx_s)$ is defined as the stochastic integral of $y$   with respect to $x$, for any $t\in [0,1]$.  
\begin{theo} \label{Girsanov}   Let $(x^\sigma_s)_{s\in [0,1]}$ be a Brownian motion on $\mathfrak k$, with variance $\sigma\in \R_+^*$, and $h\in H^1([0,1],K)$. If $\mu^\sigma$ is the law of $(\epsilon(x^\sigma)_s)_{s\in [0,1]}$, then 
$$\frac{d (\mathcal{R}_h)_*\mu^\sigma}{d\mu^\sigma}=e^{\frac{1}{\sigma}\int_0^1 ( h_s^{-1}h_s',dx^\sigma_s)-\frac{1}{2\sigma}\int_0^1( h^{-1}_sh_s',h_s^{-1}h_s')\, ds},$$
 where $(\mathcal{R}_h)_*\mu^\sigma$ is the law of $\epsilon(x^\sigma)h$.
\end{theo}
For $z\in K$, we write $\mathbb P^z$, for the probability defined on $\mathcal F_1$   as the conditional probability $\P(.\vert \epsilon(x^\sigma)_1=z)$. One has for any $s\in(0,1)$, 
\begin{align}
\mathbb P^z_{\vert\mathcal F_s}=\frac{p^\sigma_{1-s}(\epsilon(x^\sigma)_s,z)}{p_1(z)}\, .\,\mathbb P_{\vert\mathcal F_s}.
\end{align}
Note that under $\mathbb P^z$,  $(x^\sigma_s)_{s\in[0,1)}$ remains a continuous semi-martingale (see for instance \cite{Baudoin}, theorem 14), so that the stochastic integral $\int_0^t(y_s,dx^\sigma_s)$ is well defined under $\P^z$, for any $\mathfrak k$-valued $L_2$ function $y$ and $t\in(0,1)$.  Theorem \ref{Girsanov}  implies the following  proposition, which appears in proposition (5.2.12) of \cite{Frenkel}.
 
 \begin{prop}\label{condendpoint} Let $(x_s^\sigma)_{s\in [0,1]}$ be a standard Brownian motion on $\mathfrak k$ with variance $\sigma>0$ and $y$ be a $\mathfrak k$-valued $L_2$ function. 
 If $h\in H^1([0,1],K)$ satisfies   $h^{-1}h'=y$ then  for any $t\in(0,1)$,
\begin{align*}
e^{-\frac{1}{2\sigma}\int_0^t(y_s,y_s)\, ds }\E\big( e^{\frac{1}{\sigma}\int_0^{t} (y_s, dx^\sigma_s)}\vert \epsilon(x^\sigma)_1 
=z)=\frac{p_{1}^\sigma(zh_t^{-1})}{p_1^\sigma(z)}.
\end{align*}
\end{prop}
\begin{proof} For $t\in(0,1)$,
\begin{align*}
\E(e^{\frac{1}{\sigma}\int_0^t (y_s,dx_s^\sigma)}\vert \epsilon(x^\sigma)_1=z)&=\E(e^{\frac{1}{\sigma}\int_0^t (y_s,dx_s^\sigma)}\frac{p^\sigma_{1-t}(\epsilon(x^\sigma)_t,z)}{p^\sigma_1(z)})\\
&=e^{ \frac{1}{2\sigma}\int_0^t(y_s,y_s)\, ds}\E(e^{\frac{1}{\sigma}\int_0^t (y_s,dx_s^\sigma)-\frac{1}{2\sigma}\int_0^t(y_s,y_s)\, ds}\frac{p^\sigma_{1-t}(\epsilon(x^\sigma)_t,z)}{p^\sigma_1(z)})
\end{align*}
 Let $h\in H^1(K)$ such that $h^{-1}h'=y$. Theorem \ref{Girsanov}    implies  that 
\begin{align*}
\E(e^{\frac{1}{\sigma}\int_0^t (y_s,dx_s^\sigma)-\frac{1}{2\sigma}\int_0^t(y_s,y_s)\, ds}\frac{p^\sigma_{1-t}(\epsilon(x^\sigma)_t,z)}{p^\sigma_1(z)})
&= \E( \frac{p^\sigma_{1-t}(\epsilon(x^\sigma)_th_t,z)}{p^\sigma_1(z)})\\
&= \E( \frac{p^\sigma_{1-t}(\epsilon(x^\sigma)_t,zh_t^{-1})}{p^\sigma_1(z)})\\
&= \frac{p^\sigma_{1}(zh_t^{-1})}{p^\sigma_1(z)},
\end{align*} 
which gives the proposition.
% As $(\gamma^{-1}_0h\gamma)^{-1}(\gamma_0^{-1}h\gamma)'=z$, one has 
%$$
%e^{-\frac{1}{2\sigma}\int_0^t(z_s,z_s)\, ds }\E\big( e^{\frac{1}{\sigma}\int_0^{\textcolor{red}{t}} (z_s, dx^\sigma_s)}\vert \epsilon(x^\sigma)_1 
%=z)=\frac{p_{1}^\sigma(z(\gamma_0^{-1}h_t\gamma_t)^{-1})}{p_1^\sigma(z)},
%$$
%one obtains the second point letting $t$ goes to $1$.
\end{proof}
The following lemma has been proved in \cite{Frenkel}.
 \begin{lem}\label{endorbit} For $k_1,k_2\in K$, 

$$\int_Kp_s^\sigma(k_1,uk_2u^*) \,du =\sum_{\lambda\in P_+} \ch_\lambda(k_1^{-1})\ch_\lambda(k_2)  e^{-\frac{s\sigma (2\pi)^2}{2}(\vert\vert\rho+\lambda\vert\vert^2-\vert\vert \rho\vert\vert^2)}.$$

 \end{lem}
\begin{proof} For any $\lambda\in P_+$, the characters $\ch_\lambda$  satisfies
$$\int_K\frac{\ch_\lambda(k_1^{-1}uk_2u^*)}{\ch_\lambda(e)}\, du=\frac{\ch_\lambda(k_1^{-1})}{\ch_\lambda(e)}\frac{\ch_\lambda(k_2)}{\ch_\lambda(e)},$$
for every  $k_1,k_2\in K$.
Thus 
\begin{align*}
\int_K p_s^\sigma(k_1,uk_2u^*)\, du&=\sum_{\lambda\in P_+}\ch_\lambda(e)\int_K \ch_\lambda(k_1^{-1}uk_2u^*)\, du\, e^{-\frac{s\sigma (2\pi)^2}{2}(\vert\vert\rho+\lambda\vert\vert^2-\vert\vert \rho\vert\vert^2)}\\
&=\sum_{\lambda\in P_+} \ch_\lambda(k_1^{-1})\ch_\lambda(k_2)  e^{-\frac{s\sigma (2\pi)^2}{2}(\vert\vert\rho+\lambda\vert\vert^2-\vert\vert \rho\vert\vert^2)}.
\end{align*}
\end{proof}
Previous lemma and Proposition \ref{condendpoint} imply the following one.
\begin{prop}\label{condorbitendpoint} Let $(x^\sigma_s)_{s\in[0,1]}$ be a standard Brownian motion on $\mathfrak k$ starting from $0$, with variance $\sigma>0$. For    $y$ a $\mathfrak k$-valued $L_2$ function, $r\in A$, and $t\in (0,1)$,
\begin{align*} 
e^{-\frac{1}{2\sigma}\int_0^t(y_s,y_s)\, ds}\E(e^{\frac{1}{\sigma}\int_0^t(y_s,dx_s^\sigma)\, ds}\vert \rad(x^\sigma)=r)=\frac{1}{p_1^\sigma(e^r)}\sum_{\lambda\in P_+} \ch_\lambda(h_t^{-1})\ch_\lambda(e^{r})  e^{-\frac{\sigma (2\pi)^2}{2}(\vert\vert\rho+\lambda\vert\vert^2-\vert\vert \rho\vert\vert^2)}, 
\end{align*} 
where $h\in H^1([0,1],K)$ satisfies $h^{-1}h'=y$.
\end{prop} 
\begin{proof}
\begin{align*}
\E(e^{\frac{1}{\sigma}\int_0^t(y_s,dx^\sigma_s)\, ds}\vert \rad(x^\sigma)=r)&=\E(\E(e^{\frac{1}{\sigma}\int_0^t(y_s,dx^\sigma_s)\, ds}\vert \epsilon(x^\sigma)_1)\vert \rad(x^\sigma)=r)\\
&=e^{\frac{1}{2\sigma}\int_0^t(y_s,y_s)\, ds}\E( \frac{p^\sigma_{1}( \epsilon(x^\sigma)_1h_t^{-1})}{p^\sigma_1( \epsilon(x^\sigma)_1)}\vert \rad(x^\sigma)=r)\\
&=e^{\frac{1}{2\sigma}\int_0^t(y_s,y_s)\, ds}\int_K\frac{p^\sigma_1(ue^{r}u^*h^{-1}_t)}{p^\sigma_1(e^r)}\, du,
\end{align*}
which proves the proposition thanks to lemma \ref{endorbit}.
\end{proof}

\section{Affine Lie algebras and their representations}\label{affine}

In this part, we  consider an affine Lie algebra whose Cartan matrix is an extended Cartan matrix of the simple finite dimensional complex Lie algebra  $\mathfrak{g}$. Such an algebra is a nontwisted affine Lie algebra. For our purpose, we only need to consider   a realization of its Cartan subalgebra.
\subsection{Affine Lie algebras} 
 If   $\{\alpha_1,\dots,\alpha_n\}$ and  $\{\alpha^\vee_1,\dots,\alpha^\vee_n\}$ are respectively the sets of simple real roots and coroots of $K$ previously considered, we let $$\mathfrak h=\Span_\C\{\alpha^\vee_0=c-\theta^\vee,\alpha_1^\vee,\dots,\alpha^\vee_n,d\},$$
and $$\mathfrak h^*=\Span_\C\{\alpha_0=\delta-\theta,\alpha_1,\dots,\alpha_n,\Lambda_0\},$$ where
 $$\alpha_i(d)=\delta_{i0}, \quad \delta(\alpha_i^\vee)=0, \quad \Lambda_0(\alpha_i^\vee)=\delta_{i0},\quad \Lambda_0(d)=0.$$ The Killing form on $\mathfrak k$ is from now on normalized such that $(\theta^\vee,\theta^\vee)=2$. We consider its restriction to  $\mathfrak t$ and extend it to $\mathfrak h$ by $\C-$linearity, and by letting 
$$(\C c+\C d,\,\mathfrak t)=0,\quad (c,c)=(d,d)=0, \quad (c,d)=1.$$ 
The following definitions mainly come from chapters $1$ and $6$ of \cite{Kac}.  
 The linear isomorphism 
\begin{align*}
\nu:\, \,&\mathfrak{h}\to \mathfrak{h}^*, \\
& h\mapsto (h,.)
\end{align*} identifies $\mathfrak{h}$ and $\mathfrak{h}^*$. We still denote   $(.,.)$ the induced bilinear form     on $\mathfrak{h}^*$. We record that
\begin{align*}
&(\delta, \alpha_i)=0,\quad  i=0,\dots,n,\quad (\delta,\delta)=0,\quad (\delta,\Lambda_0)=1.
\end{align*}    Notice that here $\nu(\theta^\vee)=\theta$ and $(\theta^\vee\vert\theta^\vee)=(\theta\vert \theta)=2$.   
One defines the Weyl group $\widehat{W}$, as  the subgroup  of $GL(\mathfrak{h}^*)$ generated by fundamental  reflections $s_\alpha$, $\alpha\in \Pi$, defined by $$s_\alpha(\beta)=\beta-  \beta(\alpha^\vee) \alpha,\ \quad \beta\in \mathfrak{h}^*.$$   
Under the identification of $\mathfrak h$ and $\mathfrak h^*$, the action of the affine Weyl group  on $\mathfrak{h}$ is defined by $wx=\nu^{-1}w\nu x$, $x\in \mathfrak{h}$, $w\in \widehat{W}$.  The form $(., .)$ is $\widehat W$-invariant.
The affine Weyl group $\widehat W$  is the semi-direct product $W\ltimes  {\Gamma}$ (proposition 6.5 chapter 6 of \cite{Kac}) where  $\Gamma$ is the group of transformations $t_{\gamma}$, $\gamma\in \nu( {Q}^\vee)$, defined by 
$$t_\gamma(\lambda)=\lambda+\lambda(c)\gamma-((\lambda, \gamma)+\frac{1}{2}(\gamma,\gamma)\lambda(c))\delta, \quad \lambda\in \mathfrak{h}^*.$$
Thus, if $w\in W$, $\gamma\in \nu(Q^\vee)$, and $\lambda\in \mathfrak h^*$,
\begin{align}\label{actionofWaff}
wt_\gamma(\lambda)=w(\lambda+\lambda(c)\gamma)-((\lambda,\gamma)+\frac{1}{2}(\gamma,\gamma)\lambda(c))\delta.\end{align}
\begin{rem} Notice that if one identifies the quotient space $\C c\oplus\mathfrak t \oplus  d /\C c$ with $\mathfrak t$, then $\widehat W$ is identified with the extended affine Weyl group and the fundamental domain of $\C c\oplus\mathfrak t \oplus  d /\C d$ for the action of $\widehat W$ is identified with the fundamental domain of $\mathfrak t$  for the action of the extended affine Weyl group. Actually one has $$( d+y)\in\widehat{\mathcal C}\Leftrightarrow y\in A,$$ where $\widehat{\mathcal{C}}$ is the fundamental chamber defined by $$\widehat{\mathcal{C}}=\{x\in \mathfrak h: \forall \alpha\in \widehat\Sigma,\,\alpha(x)\ge 0\}.$$
The affine Weyl chamber $\widehat{\mathcal{C}}$ for the $A_1^{(1)}$ type is drawn in figure $1$ below: it is the area delimited by  yellow and orange half-planes. The alcove $A$ is highlighted in green.

\begin{center}
\includegraphics[scale=0.3]{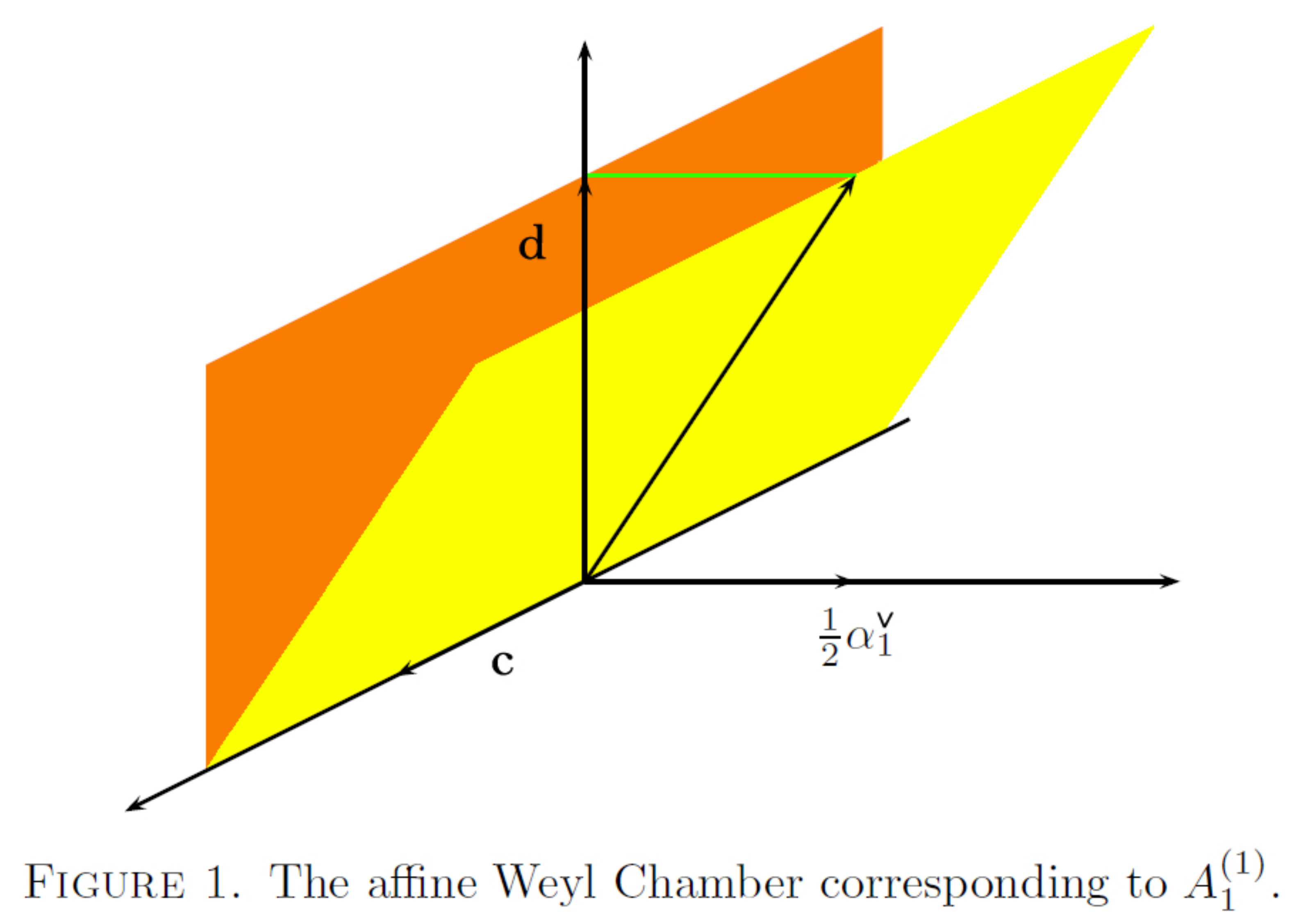}  
\end{center}

\end{rem}
\subsection{Weights, highest-weight modules, characters} The following definitions and properties mainly come from chapters $9$ and $10$ of \cite{Kac}.  We denote   $\widehat P$ (resp. $\widehat P_+$) the set of integral (resp. dominant) weights defined by 
$$\widehat P=\{\lambda\in \mathfrak{h}^*:  \langle \lambda,\alpha^\vee_i \rangle\in \Z, \, i=0,\dots, n\},$$
 $$(\textrm{resp. } \widehat P_+=\{\lambda\in \widehat P:  \langle\lambda,\alpha^\vee_i \rangle\ge 0, \, i=0,\dots,n\}),$$
where $\langle.,.\rangle$ is the pairing between $\mathfrak h$ and its dual $\mathfrak h^*$.  
For $\lambda\in \widehat P_+$ we denote   $V(\lambda)$ the irreducible module with highest weight $\lambda$.  
The Weyl-Kac character's formula (Theorem 10.4, chapter 10 of \cite{Kac}) states that 
\begin{align}\label{Weyl}
\widehat{\mbox{ch}}(V(\lambda))=\frac{\sum_{w\in \widehat W}\det(w)e^{w(\lambda+\widehat\rho)-\widehat\rho}}{\prod_{\alpha\in \widehat R_+}(1-e^{-\alpha})^{\mbox{mult}(\alpha)}.},
\end{align}
where $\widehat\rho\in \mathfrak{h}^*$ is chosen such that $\widehat\rho(\alpha_i^\vee)=1$, for all $i\in\{0,\dots,n\}$ and $\mbox{mult}(\alpha)$ is the dimension of the root space associated to the root $\alpha$.
In particular
$$\prod_{\alpha\in \widehat R_+}(1-e^{-\alpha})^{\mbox{mult}(\alpha)}=\sum_{w\in \widehat W}\det(w)e^{w(\widehat\rho)-\widehat\rho}.$$
In the sequel we choose $\widehat\rho=h^\vee\Lambda_0+\rho$. 
Writing $\widehat W =W\ltimes \Gamma$, the Weyl character's formula becomes by identity (\ref{actionofWaff})
\begin{align*}
\widehat{\mbox{ch}}(V(\lambda))=\frac{\sum\det(w)e^{w(\lambda+\widehat\rho +(\lambda+\widehat\rho )(c)\gamma) }\,e^{-((\lambda+\widehat\rho ,\gamma)+\frac{1}{2}(\gamma,\gamma)((\lambda+\widehat\rho)(c))\delta}}{\sum \det(w)e^{w( \widehat\rho  +\widehat\rho (c)\gamma) }\,e^{-(\widehat \rho ,\gamma)+\frac{1}{2}(\gamma,\gamma)\widehat \rho(c))\delta}}.
\end{align*}
where the sums run over $\gamma\in \nu(Q^\vee)$ and $w\in W$.
Letting $e^{\mu}(h)=e^{\mu(h)}$, $h\in \mathfrak{h}$, the formal character $\widehat{\mbox{ch}}(V(\lambda))$ can be seen as a function   defined on its region of convergence. Actually  the  series 
$$\sum_{\mu \in \widehat P }\mbox{dim} (V(\lambda)_\mu)e^{\langle \mu ,h\rangle} $$ converges absolutely for every $h\in \mathfrak{h}$ such that $\mbox{Re}(\delta(h))>0$. We denote by $\widehat{\mbox{ch}}_\lambda(h)$ its limit (see chapter $11$ of \cite{Kac}). The limiting function $\widehat \ch_\lambda$ is analytic on the set $$Y=\{ h\in \mathfrak{h}:\mbox{Re}(\delta(h))>0\}$$ (see chapter $12$ of \cite{Kac}).  For $y\in \mathfrak{h}$, we let $\widehat{\mbox{ch}}_y:=\widehat{\mbox{ch}}_{\nu(y)}$.
 \newline

\section{A Kirillov character type formula for affine Lie algebras}\label{Kirillov-Frenkel}
\paragraph{\bf The Kirillov's character formula for compact groups}
For $\lambda\in \mathfrak k^*$, we denote by $\mathcal{O}_{\lambda}$  the coadjoint orbit   in $\mathfrak k^*$ through $\lambda$ and by $\mu_{\lambda}$  the distribution of $\Ad(U)\lambda$ where $U$ is a random variable distributed according the Haar measure on $K$. When the restriction of $\lambda$ to $\mathfrak t$ (also denoted by $\lambda$) is a dominant weight, the Kirillov's character formula \cite{Kirillov} states  that for $x\in \mathfrak k$,    one has  
$$\int_{\mathcal{O}_{\lambda+\rho}}e^{ 2i\pi\beta(x)}\, \mu_{\lambda+\rho}(d\beta)=j(x)\frac{\ch_{\lambda}(e^x)}{\ch_\lambda(0)},$$
where
$$j(x)=\frac{\pi(x)}{h(x)},\,\pi(x)=\prod_{\alpha\in R_+}(e^{i\pi \alpha(x)}-e^{-i\pi \alpha(x)}),\, h(x)=\prod_{\alpha\in R_+} 2i\pi\alpha(x).$$ 
\begin{defn} For $x\in \mathfrak t\oplus i\mathfrak t$ and $\lambda\in \mathfrak t^*$, one defines $$\varphi_x(\lambda)=\frac{1}{h(x)}\sum_{w\in W}(-1)^w e^{\langle w \lambda,x\rangle},\quad \varphi_0(\lambda)=\frac{h(\lambda)}{h(\rho)}.$$ 
\end{defn}
The Weyl's character formula for compact Lie groups  states that 
\begin{align*}
\ch_{\lambda}(e^x)&=\frac{\sum_{w\in W}e^{2i\pi \langle w(\lambda+\rho),x\rangle}}{\sum_{w\in W}e^{2i\pi\langle w(\rho),x\rangle}} =\frac{1}{\pi(x)}
\sum_{w\in W}e^{2i\pi \langle w(\lambda+\rho),x\rangle},
\end{align*}
for any $x\in \mathfrak t$. Thus the Kirillov's character formula becomes, by approximation and analytical continuation, 
\begin{align}\label{kiriGen} 
\int_{\mathcal O_\lambda}e^{\beta(x)}\, \mu_\lambda(d\beta)=\frac{\varphi_{x}(\lambda)}{\varphi_0(\lambda)},
\end{align}
for any $x\in \mathfrak t\oplus i\mathfrak t$ and any $\lambda\in \mathfrak t^*$. This formula gives the Fourier transform of the pushforward measure  of the  measure $\mu_\lambda$ under the application $\phi\in \mathfrak k^*\mapsto \phi_{\vert \mathfrak t}\in \mathfrak t^*,$ which  is known as the Duistermaat-Heckman measure associated to $\lambda$.

\paragraph{\bf A Kirillov-Frenkel character formula for affine Lie algebras} 
Let $(x^\sigma_s)_s$ be a Brownian motion on $\mathfrak k$, with variance $\sigma>0$, and its stochastic exponential $(\epsilon(x^\sigma)_s)_s$. A Kirillov character type formula  for affine Lie algebras has been proved   by Frenkel in \cite{Frenkel}. It can be formulated as in the following theorem. 
\begin{theo}\label{kiraff} For $y\in \mathfrak t $, and $\lambda\in \widehat P_+$ such that $\lambda(c)=k$, one has  
$$\E(e^{\frac{1}{\sigma}(x^\sigma_1,y)}\vert \epsilon(x^\sigma)_1\in \mathcal O_{\exp(\sigma z)})=\widehat{j}(y)\frac{\widehat\ch_{\lambda}( d+ y)}{\widehat\ch_\lambda( d)}.$$
where $\nu(z)=\rho+\lambda-k\Lambda_0 $, $\sigma=\frac{1}{k+h^\vee}$, and  $$\widehat{j} (y)=\prod_{\alpha\in \widehat R_+\setminus R_+}\Big(\frac{1-e^{-\alpha( d+ y)}}{1-e^{-\alpha( d)}}\Big)^{\textrm{mult}(\alpha)}\prod_{\alpha\in R^+}\frac{2i\pi(1-e^{-\alpha(y)})}{1-e^{-2i\pi\alpha(y)} }.$$ 
\end{theo} 
 
\begin{defn}  For $x\in \mathfrak t$, $y\in \mathfrak t\oplus i\mathfrak t$, $a,b\in \R_+^*$, we let
\begin{align}\label{defphiaff} 
\widehat \varphi_{bd+y}(a,x)&=\frac{1}{\pi(-\frac{1}{b}y)}\sum_{w\in \widehat W}(-1)^w e^{(w(ad+x),b d+y)}
\end{align}
As $\widehat W$ is the semi-direct product $W\ltimes  {\Gamma}$ one has also
\begin{align}
 \widehat\varphi_{bd+y}(a,x)&=\frac{1}{\pi(-\frac{1}{b}y)}\sum_{\gamma\in  Q^\vee}\sum_{ w\in W}(-1)^w e^{(w(x+a\gamma),y)}e^{-b((x,\gamma)+\frac{1}{2}(\gamma,\gamma)a)},  \label{seconddef}
\end{align} 
 \end{defn}
\begin{rem} The previous definition seems  to be not valid for $y$ such that $\pi(\frac{1}{b}y)=0$. Actually,  in proposition \ref{phi-char}, we   give another expression for $\widehat\varphi_{ bd+y}$,  which shows that it is actually  well defined for every $y\in \mathfrak t\oplus i\mathfrak t$.
\end{rem}

The Poisson formula implies the following lemma.
\begin{lem} For $x\in \mathfrak t$, and $t\in \R_+^*$, one has
\begin{align*}
\sum_{\mu\in P}e^{2i\pi\mu(x)-\frac{t}{2}(2\pi)^2(\mu,\mu)}&=C\, (\frac{1}{2\pi t})^{n/2} \sum_{z\in Q^\vee}e^{-\frac{1}{2t}(x+ z,x+ z)}.
\end{align*}
where $C$ is a positive constant independent of $x$ and $t$.
\end{lem}
\begin{proof} As $\mu$ is a weight in $P$ if and only if  $\mu(Q^\vee)\subset\Z,$ the Poisson summation formula gives for every bounded continuous function $f$ defined on $\mathfrak t^*$,
$$\sum_{\mu\in P}f(\mu)=C\sum_{z\in Q^\vee}\hat f(z),$$
where $\hat f(z)=\int_{\mathfrak t^*}e^{2i\pi \mu(z)}\, f(\mu)\, d\mu$, and $C$ is a positive constant independent of $f$. The lemma follows taking $f$ defined by $f (\mu)=e^{2i\pi\mu(x)-\frac{t}{2}(2\pi)^2(\mu,\mu)}$, $\mu\in \mathfrak t^*$.
\end{proof}
\begin{prop}\label{phi-char} For $x,y\in \mathfrak t$, and $\sigma\in \R_+^*$, one has 
\begin{align*}  
\widehat\varphi_{ d+y}(\frac{1}{\sigma},\frac{x}{\sigma})&=C(\frac{1}{2\pi\sigma})^{-n/2}e^{\frac{1}{2\sigma}(y,y)+\frac{1}{2\sigma}(x,x)}\sum_{\mu\in P_{+}}\pi(x)\ch_{\mu}(e^x)\ch_{\mu}(e^{-y})e^{-\frac{\sigma}{2}(2\pi)^2\vert\vert\mu+\rho\vert\vert^2},
\end{align*} 
where $C$ is a positive constant independent of $x$, $y$ and $\sigma$. 
\end{prop}
\begin{proof}
One has to  prove that
$$
\frac{1}{\pi(-y)}\sum_{w\in   \widehat{W}} (-1)^we^{\frac{1}{\sigma}(w( d+y), d+x)}$$ 
is equal to $$ C(\frac{1}{2\pi\sigma})^{-n/2}\, e^{\frac{1}{2\sigma}(y,y)+\frac{1}{2\sigma}(x,x)}\sum_{\mu\in P_{+}}\pi(x)\ch_{\mu}(e^x)\ch_{\mu}(e^{-y})e^{-\frac{\sigma}{2}(2\pi)^2\vert\vert\mu+\rho\vert\vert^2},
$$
where $C$ is a positive constant independent of $x$, $y$ and $\sigma$.  
 On the one hand, on has
\begin{align*}
C\sum_{\mu\in P,w\in W}(-1)^w &e^{2i\pi\mu(x-wy)-\frac{t}{2}(2\pi)^2(\mu,\mu)}=(\frac{1}{2\pi t})^{n/2}\sum_{z\in Q^\vee,w\in W}(-1)^we^{-\frac{1}{2t}(x-wy+ z,x-wy+ z)}\\
&= (\frac{1}{2\pi t})^{n/2} e^{-\frac{1}{2t}(y,y)-\frac{1}{2t}(x,x)}\sum_{w\in   W,z\in Q^\vee} (-1)^we^{\frac{1}{t}(w(y),x+ z)-\frac{1}{t}((x,z)+\frac{1}{2}(z,z))}\\
&=(\frac{1}{2\pi t})^{n/2} e^{-\frac{1}{2t}(y,y)-\frac{1}{2t}(x,x)}\sum_{w\in   \widehat W} (-1)^we^{\frac{1}{t}(w( d+y), d+x)}.
\end{align*} 
On the other hand,
\begin{align*}
\sum_{\mu\in P,w\in W}(-1)^w e^{2i\pi\mu(x-wy)-\frac{t}{2}(2\pi)^2(\mu,\mu)}&=\sum_{\mu\in P_{+},w,\tilde w\in W}(-1)^w e^{2i\pi\tilde w(\mu+\rho)(x-wy)-\frac{t}{2}(2\pi)^2(\tilde w (\mu+\rho),\tilde w(\mu+\rho))}\\
&=\sum_{\mu\in P_{+}}\ch_{\mu}(e^x)\pi(x)\ch_{\mu}(e^{-y})\pi(-y)e^{-\frac{t}{2}(2\pi)^2\vert\vert\mu+\rho\vert\vert^2}
\end{align*}
\end{proof} 
\begin{rem} By analytical continuation, the identity remains true for $y\in \mathfrak t\oplus i\mathfrak t$, $\sigma\in \mathbb C$, $\mbox{Re}(\sigma)>0$.
\end{rem}
Propositions \ref{phi-char} and \ref{condorbitendpoint} imply the following theorem, which implies the Kirillov character type formula of theorem \ref{kiraff}. 
\begin{theo} \label{endpoint} For $y\in L_2([0,1],\mathfrak k)$,   and $z\in   A$,  one has for $t\in (0,1)$,
$$\E(e^{\frac{1}{\sigma}\int_0^t(y_s,dx_s^\sigma)}\vert \rad(x^\sigma)=z)=e^{\frac{1}{2\sigma}\int_0^t (y_s,y_s)\, ds}e^{-\frac{1}{2\sigma}(a,a)}\frac{\widehat\varphi_{ d+a}(\frac{1}{\sigma},\frac{z}{\sigma})}{\widehat\varphi_{ d}(\frac{1}{\sigma},\frac{z}{\sigma})},$$
where $h^{-1}h'=y$ and $h_t\in \mathcal O_{e^a}$, with $a\in \mathfrak t$. In particular for $y\in \mathfrak t$, one has
$$\E(e^{\frac{1}{\sigma}(y,x_1^\sigma)}\vert \rad(x^\sigma)=z)= \frac{\widehat{\varphi}_{ d+y}(\frac{1}{\sigma},\frac{z}{\sigma})}{\widehat{\varphi}_{ d}(\frac{1}{\sigma},\frac{z}{\sigma})}.$$

\end{theo}  
 \begin{rem}  Let us make some non rigorous remarks about these formulae. Previously we have considered the identification of $L_2([0,1],\mathfrak k)'\oplus \R\Lambda_0$ with $H_1([0,1],\mathfrak k)\oplus \R\Lambda_0$ letting $\varphi_x=\int_0^1(.,x'_s)\, ds$ for $x\in H_1([0,1],\mathfrak k)$. In the first formula the stochastic integral $\int_0^1 (.,dx^\sigma_s)$ can be seen as a random linear functional. Its conditional law  given $\epsilon(x^\sigma)\in \mathcal O_{e(  z)}$ has to be thought as a measure on a coadjoint orbit through $\varphi_{\pi_{  z}}+\Lambda_0$ where $\pi_{  z}$ is the straight path $\pi_{  z}(s)=s  z$, $s\in[0,1]$.
For the second formula, we notice that the restriction of $\varphi_x$ to $\mathfrak t$ is equal to $y\in \mathfrak t\mapsto (y,x_1)$. Thus the law of $x_1^\sigma$ given $\epsilon(x^\sigma)\in \mathcal O_{e(  z)}$ has to be thought as a Duistermaat-Heckman distribution associated to $\varphi_{\pi_{  z}}+\Lambda_0$.  
\end{rem}
  \section{A Brownian motion on $\R d \oplus \mathfrak t$ conditioned to remain on the affine Weyl chamber}\label{h-process} \label{cond-affine}
 Let us consider a standard Brownian motion $(b_t)_{t\ge 0}$ on $   \mathfrak{t}$.
 We consider a random process $(\tau_t,b_t)_{t\ge 0}$ on $  \R \times \mathfrak{t}$ as follows. For     $u_0\in \R,x_0 \in  \mathfrak{t}$, denote    $\mathbb{W}_{u_0,x_0}$  a probability under which $\tau_t=u_0+ t,$ for all $ t\ge 0$, and  $(b_t)_{t\ge 0}$ is a standard Brownian motion    starting from $x_0$.  We let $x_t=\tau_td+b_t$, for $t\ge 0$. The process $(x_t)_{t\ge 0}$  is a space-time Brownian motion, $\tau_t$ being the time component.
Let us consider the stopping time $$T=\inf\{t\ge 0:  x_t\notin \widehat{\mathcal C}\}=\inf\{t\ge 0: b_t\notin \tau_t A\}.$$
\begin{lem}\label{mart} Let $u>0$ and $(b_t)_{t\ge 0}$ be a standard Brownian motion on $\mathfrak t$. 
For $y\in \mathfrak t \oplus i\mathfrak t$, $(e^{-\frac{(y,y)t}{2}}\widehat \varphi_{ d+y}(\tau_t,b_t),t\ge 0)$ is a true martingale under $\mathbb W_{x,u}$.
\end{lem}
\begin{proof} Let us suppose that $\pi(y)\ne 0$.   Choose an orthonormal basis $v_1,\dots,v_n$ of $\mathfrak{t}$ and consider    for $w\in \widehat W $ a function $g_w$ defined on $\R_+^*\times \R^{n}$ by $$g_w(t,x_1,\dots,x_n)=e^{( t d+x,w( d+iy))+\frac{t}{2}(y,y)},$$ 
where $x=x_1v_1+\dots+x_nv_n$. Letting $\Delta=\sum_{i=1}^n\partial_{x_ix_i}$,   the function $g_w$ satisfies 
\begin{align}\label{edp}(\frac{1}{2}\Delta+\partial_t) g_w= 0,\end{align}
which implies that $(g_w(t+u,b_t))_{t\ge 0}$ is a local martingale.  Its quadratic variation is easily shown to be integrable, so that, it is a true martingale. The sum $\sum_{w\in \widehat{W}}(-1)^wg_w(t+u,b_t)$ converges   for the $L^1$ norm, which implies that $$(e^{\frac{t}{2}(y,y)}\varphi_{ d+y}(t+u,b_t),t\ge 0),$$ is also a true martingale. Continuity in $y$ ensures that  lemma is true for any $y\in \mathfrak t \oplus i\mathfrak t$.
\end{proof}
\begin{lem} Let $t>0$. If $td+x\in \widehat{\mathcal C}$, i.e. $x\in tA$, then $\widehat\varphi_{ d}( t,x)\ge 0,$
with equality if and only if   $td+x$ is on the boundary of $\widehat{\mathcal C}$, i.e. $x$ is on the boundary of $t  A$.
\end{lem}
\begin{proof}
Proposition \ref{phi-char}  implies  that for any $z\in \mathfrak k$,
 
$$\widehat\varphi_{ d}(\frac{1}{\sigma},\frac{z}{\sigma})=C \times p_1^\sigma(e^z)\pi(z)e^{\frac{1}{2\sigma}(2\pi)^2\vert\vert z\vert\vert^2-\frac{\sigma}{2}(2\pi)^2\vert\vert \rho\vert\vert^2},$$
where $C$ is a positive constant, which ensures in particular that 
$$\widehat\varphi_{d}( t,z)\ge 0,$$
for $z\in tA$, with equality if and only if $z\in t\partial A$. 
\end{proof} 
Let $(\mathcal F_t)_{t\ge 0}$ be the natural filtration of $(x_t)_{t\ge 0}$.   As $(\widehat{\varphi}_{ d}(\tau_{t\wedge T},b_{t\wedge T}),t\ge 0),$ is a positive true martingale under $\mathbb{W}_{x,u}$ such that  $\widehat\varphi_{ d}( T,b_{T})=0$, one defines a measure $\mathbb{Q}_{x,u}$ on $\mathcal F_\infty$ as below.
\begin{defn} Let $u>0$ and  $ x\in  \mathfrak{t}$ such that $ ud+x$ is  in the interior of $\widehat{\mathcal C}$, i.e. $\frac{x}{u}$ in the interior of $A$. One defines a probability $\Q_{u,x}$ letting 
$$\Q_{u,x}(B)=\E_{\mathbb{W}_{u,x}}(\frac{\widehat\varphi_{ d}(\tau_t,b_t)}{\widehat\varphi_{ d}( u,x)}1_{T\ge t,\, B}), \textrm{ for } B\in \mathcal F_t,\, t\ge 0.$$
\end{defn}  
Actually  this conditioned Doob process  can be obtained as a limit of a space time Brownian motion starting in the affine Weyl chamber, with drift  within the affine  Weyl chamber,  conditioned to remain forever in it, when the drift goes to zero  (see \cite{Defosseux1} and \cite{Defosseux2} for more details). Lemma \ref{mart} and the fact that for  $y\in \mathfrak t\oplus i\mathfrak t$, $\widehat\varphi_{d+y}(\tau_T,b_T)=0$ imply immediately the following proposition.
\begin{prop}\label{phiQ} For $r,t,u\in \R_+^*$, $x$ in the interior of $u{A}$, and $y\in \mathfrak t\oplus i\mathfrak t$, one has 
\begin{align}\label{continuous}
\E_{\mathbb{Q}_{u,x}}(\frac{\widehat\varphi_{d+y}(\tau_t,b_t)}{\widehat\varphi_{d}(\tau_t,b_t)})=\frac{\widehat\varphi_{d+y}(u,x)}{\widehat\varphi_{d}( u,x)}e^{\frac{(y,y)}{2}t}.
\end{align}
and
\begin{align}\label{continuouscond}
E_{\mathbb{Q}_{x,u}}(\frac{\widehat\varphi_{d+y}( \tau_{r+t},b_{t+r})}{\widehat\varphi_{d}( \tau_{t+r},b_{t+r})}\vert \mathcal{F}_r)=\frac{\widehat\varphi_{ d+y}(\tau_r,b_r)}{\widehat\varphi_{d}(\tau_r,b_r)}e^{\frac{(y,y)}{2}t}.
\end{align}
\end{prop} 
\section{Conditioned space time brownian motion, and radial part of a Brownian sheet}\label{cond-rad}
 In this last section we prove that the  conditioned Doob process in an affine Weyl chamber previously introduced, has the same law as the radial part process of a Brownian sheet on $\mathfrak k$. It is stated in theorem \ref{maintheo}.   Let $(x_s^t)_{s\in [0,1],t\in \R_+}$ be  a standard  Brownian sheet  on $\mathfrak k$, i.e. for any $t,t'\in \R$, $(x_s^{t'+t}-x_s^{t})_{ s\in [0,1]}$ is a $\mathfrak k$-valued  random process independent of $\sigma(x_s^r, s\in [0,1],r\le t)$, having the  same law as $(x_s^{t'})_{s\in [0,1]}$, which is a standard Brownian motion on $\mathfrak k$, with variance $t'$.    In the sequel we choose a continuous version of it.
Proposition \ref{pointconvergence} and corollary \ref{entrancelaw} prove the existence of an entrance law for the conditioned process in the affine Weyl chamber introduced in section \ref{h-process}, and the entrance point $0$.

\begin{lem} \label{convergence} For any $y\in \mathfrak t\oplus i\mathfrak t$,
$$\frac{\widehat\varphi_{ d+y}(u,x)}{\widehat\varphi_{ d}(u,x)}$$
converges towards $1$ when  $(u,x)$ goes to $0$ within the affine Weyl chamber.
\end{lem}
\begin{proof}
\begin{align*}  
\widehat\varphi_{ d+y}(u,x)&=c(\frac{ u}{2\pi})^{-n/2}e^{\frac{u}{2 }(y,y)+\frac{1}{2u}(x,x)}\pi(\frac{x}{u})\sum_{\mu\in P_{+}}\ch_{\mu}(e^{\frac{x}{u}})\ch_{\mu}(e^{-y})e^{-\frac{1}{2u}(2\pi)^2\vert\vert\mu+\rho\vert\vert^2},
\end{align*} 
Lemma 13.13 of \cite{Kac} implies that the dominant term in the sum of the right hand side of the identity is 
 $e^{-\frac{1}{2u}(2\pi)^2\vert\vert\rho\vert\vert^2},$
and proposition follows.
\end{proof}

\begin{lem}\label{conv1} Let us fix $t>0$. If $\mu_{x,u}$ is the law of $\frac{b_t}{t+u}$ under $\Q_{x,u}$ then for any $y\in \mathfrak t\oplus i\mathfrak t$,
$$\int_A\frac{\widehat\varphi_{ d+y}(t,tz)}{\widehat\varphi_{ d}(t,tz)}\, \mu_{x,u}(dz)$$
converges towards $e^{t\frac{(y,y)}{2}}$ when $(x,u)$ goes to $0$ within the affine Weyl chamber.
\end{lem}
\begin{proof}If $\mu_{x,u}$ is the law of $\frac{b_t}{t+u}$ under $\mathbb Q_{x,u}$ then 
$$\int_A\frac{\widehat\varphi_{d+y}(t+u,(t+u)z)}{\widehat\varphi_{d}(t+u,(t+u)z)}\, d\mu_{x,u}(z)=\frac{\widehat\varphi_{ d+y}(x,u)}{\widehat\varphi_{ d}(x,u)}e^{\frac{t}{2}(y,y)}.$$
Proposition \ref{phi-char} and the fact that $\vert ch_\mu(e^z)\vert \le\frac{h(\mu+\rho)}{h(\rho)}$ imply that   
$\frac{\widehat\varphi_{d+y}((t+u),(t+u)z)}{\widehat\varphi_{d}(t+u,(t+u)z)}$ converges to $\frac{\widehat\varphi_{ d+y}(t,tz)}{\widehat\varphi_{d}(t,tz)}$, uniformly in $z\in A$, when $u$ goes to $0$. Thus lemma follows from lemma \ref{convergence} 
\end{proof}

\begin{prop}\label{pointconvergence}
Let $(x_{s}^t)_{s\in[01],t\in \R_+}$ be a standard Brownian sheet on $\mathfrak k$ and $t>0$. Under $\mathbb{Q}_{u,x}$, $\frac{b_t}{t+u}$ converges in law towards the  radial part of $(\frac{1}{t}x_{s}^t)_{s\in [0,1]}$, when $(x,u)$ goes to $0$ within the affine Weyl chamber.
\end{prop}
\begin{proof} Let $t>0$. By analytical continuation, the second identity of theorem \ref{endpoint} remains valid for any $y\in \mathfrak t\oplus i\mathfrak t$. As $(\frac{1}{t}x_{s}^t)_{s\in [0,1]}$ is a standard Brownian motion on $\mathfrak k$ with variance $\frac{1}{t}$, theorem \ref{endpoint} gives for $z\in A$, $y\in\mathfrak t\oplus i\mathfrak t$,
$$\E(e^{(x_{1}^t,y)}\vert \rad((\frac{x_{s}^t}{t})_s)=z)=\frac{\widehat\varphi_{d+y}(t,tz)}{\widehat\varphi_{d}(t,tz)}.$$
In particular  the law $\mu$ of $\rad((\frac{x_{s}^t}{t})_s)$  satisfies, for any $y\in\mathfrak t\oplus i\mathfrak t$,
$$\int_A\frac{\widehat\varphi_{d+y}(t,tz)}{\widehat\varphi_{d}(t,tz)}\, d\mu(z) =\E(e^{(x_{1}^t,y)})=e^{\frac{t}{2}(y,y)}.$$Proposition \ref{phi-char} and the fact that $\vert ch_\mu(e^z)\vert \le\frac{h(\mu+\rho)}{h(\rho)}$   imply      that 
$$z\in A\mapsto \frac{\widehat\varphi_{d}(t,zt)}{\pi(z)e_t(z)},$$ 
where $e_t(z)=\sum_{\gamma\in \nu(Q^\vee)} e^{-t(\gamma(z)+\frac{1}{2}(\gamma,\gamma))}$, is smooth on $A$. Thus the Peter-Weyl theorem ensures that for any smooth function $u$ defined on $A$, and $z\in A$, 
$$ u(z)\frac{\widehat\varphi_d(t,zt)}{\pi(z)e_t(z)}=\sum_{\lambda\in P_+}c_\lambda \ch_\lambda(z),$$
where $c_\lambda=\int_Au(z)\frac{\widehat\varphi_d(t,zt)}{\pi(z)e_t(z)}\ch_\lambda(z)\pi^2(z)\, dz$, and the convergence stands uniformly and absolutely. Actually $\lim_{\vert\lambda\vert\to\infty}(\lambda,\lambda)^nc_\lambda=0$ for all $n\in \N$. As  $\frac{\widehat\varphi_d(t,zt)}{\pi(z)e(z)}$ remains positive on $A$,
$$ u(z)=\sum_{\lambda\in P_+}c_\lambda \ch_\lambda(z)\frac{\pi(z)e_t(z)}{\widehat\varphi_d(t,zt)}.$$
As $\vert \ch_\lambda(z)\vert\le d(\lambda+\rho)$ and $\pi(.)e(.)/\widehat\varphi_d(t,t.)$ is bounded on $A$, the uniform and absolute convergence implies that for any probability measure $\nu$ on $A$, 
\begin{align}\label{idnu}
\int_Au(z)\nu(dz)=\sum_{\lambda\in P_+}c_\lambda\int_A \ch_\lambda(z)\frac{\pi(z)e_t(z)}{\widehat\varphi_d(t,zt)}\nu(dz).
\end{align}
Expression (\ref{seconddef}) for the definition of  $\widehat\varphi_{ d+y}$ gives, $$\widehat\varphi_{ d+2\pi\frac{i}{t}(\lambda+\rho)}(t,tz)=\frac{1}{\pi(-2\pi\frac{i}{t}(\lambda+\rho))}\ch_\lambda(z)\pi(z)e_t(z),$$ for any $\lambda\in P_+$. Thus one has 
\begin{align*}
\int_A \ch_\lambda(z)\frac{\pi(z)e(z)}{\widehat\varphi_d(t,zt)}d\mu(z)
&=\pi(-2\pi\frac{i}{t}(\lambda+\rho))\int_A\frac{\widehat\varphi_{d+2\pi\frac{i}{t}(\lambda+\rho)}(t,tz)}{\widehat\varphi_{d}(t,tz)}\, d\mu(z) \\
&=\pi(-2\pi\frac{i}{t}(\lambda+\rho))e^{-\frac{t}{2}(2\pi)^2(\lambda+\rho,\lambda+\rho)},
\end{align*}
and identity (\ref{idnu})  becomes for $\nu=\mu$,
$$\int_A u(z)\, d\mu(z)=\sum_{\lambda\in P_+}c_\lambda \pi(-2\pi\frac{i}{t}(\lambda+\rho))e^{-\frac{t}{2}(2\pi)^2(\lambda+\rho,\lambda+\rho)}.$$ 
Identity (\ref{idnu})  becomes for $\nu=\mu_{x,u}$ 
$$\int_A u(z)\, d\mu_{x,u}(z)=\sum_{\lambda\in P_+}c_\lambda  \int_A  \ch_\lambda(z)\frac{\pi(z)e(z)}{\widehat\varphi_d(t,zt)}\, d\mu_{x,u}(z).$$
As
\begin{align*}
\int_A \ch_\lambda(z)\frac{\pi(z)e(z)}{\widehat\varphi_d(t,zt)}d\mu_{x,u}(z)
&=\pi(-2\pi\frac{i}{t}(\lambda+\rho))\int_A\frac{\widehat\varphi_{ d+2\pi\frac{i}{t}(\lambda+\rho)}(t,tz)}{\widehat\varphi_{d}(t,tz)}\, d\mu(z) 
\end{align*}
lemma \ref{conv1} implies that
$$\int_Au(z)d\mu_{x,u}(z)\textrm{ converges towards } \int_Au(z)d\mu(z),$$
as $(x,u)$ goes to $0$ within the affine Weyl chamber, which proves the proposition.
\end{proof}
\begin{cor}  \label{entrancelaw} For any $t>0$, 
$$\lim_{(x,u)\to 0}\Q_{x,u}(b_t\in dz)=C_t\widehat\varphi_{ d}(t, z)\pi(\frac{z}{t})1_A(\frac{z}{t})\mathbb W_0(b_t\in dz),$$
when $(x,u)$ goes to $0$ within the affine Weyl chamber.
\end{cor}
\begin{proof} The Weyl integration formula implies that the density of the radial part of $\frac{1}{t}x^t$ is equal to $p^{\frac{1}{t}}_1( {z})\pi( {z})^21_A({z})$, up to a multiplicative constant.  Thus corollary follows from propositions \ref{pointconvergence} and  \ref{phi-char}.
\end{proof}
One defines a law $\mathbb Q_{0+}$ on $\sigma(x_u : u>0)$ letting for $B\in \sigma(x_u: u\ge t)$, $t>0$,
\begin{align}\label{Q0+}
\mathbb Q_{0+}(B)=\mathbb P_0(C_t\widehat\varphi_{ d}( t, b_t)\pi(\frac{b_t}{t}) \mathbb Q_{x_t}(\theta_t B)),
\end{align}
where $\theta$ is the shift operator.
\begin{rem} Under $\mathbb Q_{0+}$, $\frac{b_t}{t}$ is equal in law to $\rad(\frac{x^t}{t})$ for any $t>0$. Thus for any $t>0$, $\frac{b_t}{t}$ is equal in law to $\rad(x^{\frac{1}{t}})$. As $\epsilon(x^{\frac{1}{t}})_1$ converges towards the Haar measure on $K$ when $t$ goes to $0$,  one obtains by the Weyl integration formula that 
$$\lim_{t\to 0}\mathbb{Q}_{0^+}(\frac{b_t}{t}\in dx)=C\pi(x)^21_A(x)dx,$$
which can be also deduced from (\ref{Q0+}).
\end{rem}
Let us denote by $C([0,1],\mathfrak k)$ the set of continuous maps from $[0,1]$ to $\mathfrak k$.
\begin{prop} \label{intertwining} Let $(x_s^t)_{s\in [0,1],t\in \R_+}$ be a standard Brownian sheet on $\mathfrak k$.  Let $\Lambda$ be a kernel on $( \R_+^*\times A)\times (\R_+ \times C([0,1],{\mathfrak k} ))$, such that for any $(t,y)\in \R_+^*\times A$, $\Lambda((t,y),(t,.))$ is the law of a brownian motion $(x^t_s)_{s\in[0,1]}$ on $\mathfrak k$, given $\rad(\frac{1}{t}x^t)=y$. Let $(P_t)_{t\in \R_+^*}$ be the transition probability of the Markov process $(\tau_t,\frac{b_t}{\tau_t})_{t\in \R_+^*}$ under $\mathbb{Q}_{0+}$, and  $(Q_t)_{t\in \R_+^*}$ the transition probability of $(x^t,t)_{t\in\R_+}$. Then for any   $t>0$ and $(u_0,x_0)\in \R_+^*\times A$, 
$$P_t\Lambda((u_0,x_0),.)=\Lambda Q_t((u_0,x_0),.).$$     
\end{prop}
\begin{proof}   It is sufficient to prove that for every measurable fonction $y\in L_2([0,1],\mathfrak k)$, $(u_0,x_0)\in \R_+^*\times A$, and $r\in(0,1)$, $$\int P_t\Lambda((u_0,x_0),(t+u_0,dz))e^{\int_0^r(y_s,dz_s)}=\int \Lambda Q_t((u_0,x_0),(t+u_0,dz))e^{\int_0^r(y_s,dz_s)}.$$
Let $y\in L_2([0,1],\mathfrak k)$, $(u_0,x_0)\in \R_+^*\times A$,  $r\in(0,1)$,  $h\in H^1([0,1],K)$ such that $h^{-1}h'=y$, and $a\in \mathfrak t$ such that $h_r\in \mathcal O_{e^a}$. On the one hand, one has
\begin{align*} 
\int P_t&\Lambda((u_0,x_0),(t+u_0,dz))e^{\int_0^r(y_s,dz_s)}\\
&=\int P_t((u_0,x_0),(u_0+t,dz))\Lambda((u_0+t,z),(u_0+t,d\tilde z))e^{\int_0^r(y_s,d\tilde z_s)}\\
&=\int P_t((u_0,x_0),(u_0+t,dz))e^{\frac{1}{2}(t+u_0)\int_0^r(y_s,y_s)\, ds-\frac{1}{2}(t+u_0)(a,a)}\frac{\widehat\varphi_{ d+a}(t+u_0,(u_0+t)z)}{\widehat\varphi_{ d}(t+u_0,(u_0+t)z)}\\
&=e^{\frac{1}{2}(t+u_0)\int_0^r(y_s,y_s)\, ds-\frac{1}{2}(t+u_0)(a,a)}\E_{\mathbb Q_{u,ux}}(\frac{\widehat\varphi_{ d+a}(t+u,b_{t})}{\widehat\varphi_{d}(t+u,b_{t})})\\
&=e^{\frac{1}{2}(t+u_0)\int_0^r(y_s,y_s)\, ds-\frac{1}{2}u_0(a,a)}\frac{\widehat\varphi_{ d+a}(u,ux)}{\widehat\varphi_{ d}(u,ux)}.
\end{align*}
On the other hand,
\begin{align*}  
\int \Lambda(&(x,u_0),(dz,u_0))Q_t((z,u),(d\tilde z,u+t))e^{\int_0^r(y_s,d\tilde z_s)}\\
&=e^{\frac{t}{2} \int_0^r(y_s,y_s)\, ds}\int    \Lambda((x,u),(dz,u))e^{\int_0^r(y_s,dz_s)}\\
&=e^{\frac{t}{2} \int_0^r(y_s,y_s)\, ds} e^{\frac{u_0}{2}\int_0^r(y_s,y_s)\, ds-\frac{1}{2}u_0(a,a)}\frac{\widehat{\varphi}_{ d+a}(u,ux)}{\widehat{\varphi}_{ d}(u,ux)}.
\end{align*}
\end{proof}
For the following theorem we consider a continuous version of the doubly  indexed process $(\epsilon(\frac{x^t}{t})_s)_{s\in [0,1],t\in\R_+^*}$ (see for instance \cite{DriveretAl} and references therein for the existence of a continuous version).
\begin{theo}\label{maintheo}
Let $(x_{s}^t)_{s\in[01],t\in \R_+}$ be a standard Brownian sheet. Under $\mathbb Q_{0+}$, the process $(\frac{b_t}{t})_{t\in \R_+^*}$  is equal  in law to the process $\big(\rad(\frac{x^t}{t})\big)_{t\in \R_+^*}$.
\end{theo}
\begin{proof}  Applying the criterion described by Rogers and Pitman in \cite{Pitman}, the theorem is a consequence of the proposition \ref{pointconvergence} and the intertwining relation established in proposition \ref{intertwining}.
\end{proof}

\end{document}